\UseRawInputEncoding

\documentclass[12pt]{amsart}

\usepackage{amsmath,amsthm,verbatim,amssymb,amsfonts,amscd, graphicx,cite,bbm,url,enumerate,hyperref,comment}
\usepackage{graphics}
\hypersetup{
	colorlinks=true,
	linkcolor=blue,
	citecolor=black
}
\theoremstyle{plain}
\newtheorem{theorem}{Theorem}[section]
\newtheorem{corollary}[theorem]{Corollary}
\newtheorem{lemma}[theorem]{Lemma}
\newtheorem{proposition}[theorem]{Proposition}

\theoremstyle{definition}

\theoremstyle{definition}

\allowdisplaybreaks[2]
\numberwithin{equation}{section}

\newcommand{\ex}{\mathbb{E}}

\newcommand{\re}{\textup{Re}}

\begin{document}

\title{Low pseudomoments of Euler products} 

\author{Maxim Gerspach}
\address{Department of Mathematics, KTH Royal Institute of Technology, Lindstedtsvägen 25, 114 28 Stockholm, Sweden}
\email{gerspach@kth.se}

\author{Youness Lamzouri}
\address{Institut \'Elie Cartan de Lorraine, Universit\'e de Lorraine, BP 70239, 54506 Vandoeuvre-l\`es-Nancy Cedex, France}

\email{youness.lamzouri@univ-lorraine.fr}

\thanks{The first author is partially supported by Swedish Research Council Grant No. 2016-05198.}

\date{\today}

\begin{abstract} In this paper, we determine the order of magnitude of the $2q$-th pseudomoment of powers of the Riemann zeta function $\zeta(s)^{\alpha}$ for $0<q\le 1/2$ and $0< \alpha<1$, completing the results of Bondarenko, Heap and Seip, and of Gerspach. Our results also apply to more general Euler products satisfying certain conditions. 

\end{abstract}


\maketitle

\section{Introduction}

A fundamental problem in analytic number theory is to obtain asymptotic formulas for the moments of the Riemann zeta function. Motivated by this longstanding problem, Conrey and Gamburd \cite{CG} defined and studied its \emph{pseudomoments}, which correspond to the moments of its partial sums on the critical line. More precisely, the $2q$-th pseudomoment of the zeta function is defined by
\begin{equation}\label{PseudoZeta}
\Psi_{2q}(x):=\lim_{T\to \infty} \frac{1}{T} \int_0^T \left|\sum_{n \le x} \frac{1}{n^{1/2+i t}}\right|^{2 q} \, dt,\end{equation}
where $q>0$ is a real number. When $q$ is a positive integer, Conrey and Gamburd \cite{CG} proved the following asymptotic formula
$$ 
\Psi_{2q}(x) \sim C_q (\log x)^{q^2},
$$
where $C_q= a_q \gamma_q$, with $a_q$ being the arithmetic factor in the conjecture of Keating and Snaith \cite{KS} for the $2q$-th moment of the Riemann zeta function, and $\gamma_q$ is a geometric factor which corresponds to the volume of a certain convex polytope.

In \cite{BHS}, Bondarenko, Heap and Seip investigated the pseudomoments $\Psi_{2q}(x)$ for non-integral values of $q$. In particular, they proved that when $q>1/2$ is a real number one has 
\begin{equation}\label{OrderPZ}
\Psi_{2q}(x) \asymp (\log x)^{q^2}.
\end{equation}
They also obtained the corresponding  lower bound for all $q>0$ 
$$ \Psi_{2q}(x) \gg (\log x)^{q^2}. $$
The problem of obtaining the correct order of magnitude for the low pseudomoments $\Psi_{2q}(x)$ when $0<q\leq 1/2$ was solved in a recent work of Gerspach \cite{Ge}, who proved that \eqref{OrderPZ} is valid in the whole range $q>0$. 

In \cite{BBSSZ}, Bondarenko,  Brevig, Saksman, Seip, and Zhao  investigated the related problem of obtaining the order of magnitude for the pseudomoments of powers of the Riemann zeta function, and uncovered an interesting phenomenon in this case. For a complex number $\alpha$ we have
$$ \zeta(s)^{\alpha}= \sum_{n\ge 1} \frac{d_{\alpha}(n)}{n^s},$$
for $\re(s)>1$, where $d_{\alpha}$ is the generalized divisor function. The $2q$-th pseudomoment of $\zeta(s)^{\alpha}$ is defined similarly to \eqref{PseudoZeta} by 
$$ \Psi_{2q, d_{\alpha}}(x):=\lim_{T\to \infty} \frac{1}{T} \int_0^T \left|\sum_{n \le x} \frac{d_{\alpha}(n)}{n^{1/2+i t}}\right|^{2 q} \, dt.
$$
Similarly to the estimate \eqref{OrderPZ}, and in analogy with the conjectured asymptotics for the moments of the zeta function, one might guess that for $q>0$ and $\alpha >0$ we have
\begin{equation}\label{PseudoPowerZeta}
\Psi_{2q, d_{\alpha}}(x) \asymp (\log x)^{(q\alpha)^2}.
\end{equation}
Bondarenko, Heap and Seip \cite{BHS} proved that this is the case in the range $ q> 1/2$ for all $\alpha>0$. However, the authors of \cite{BBSSZ} showed that this order of magnitude cannot be correct when $\alpha>1$ for small values of $q$. Improving on their results by building on work of Harper \cite{Ha2}, Gerspach \cite{Ge} determined the correct order of magnitude of $\Psi_{2q, d_{\alpha}}(x)$ up to powers of $\log\log x$, when $0<q<1/2$ and $\alpha\geq 1$ are fixed. More precisely, he showed that 
\begin{equation}\label{Maxim}
\Psi_{2q, d_{\alpha}}(x) = \begin{cases} (\log x)^{2(\alpha-1) q} (\log\log x)^{O(1)} & \text{ if } 1\leq \alpha <2 \text{ and }  0 < q \le \frac{2(\alpha-1)}{\alpha^2},
\\  (\log x)^{(q\alpha)^2}& \text{ if } 1\leq \alpha <2 \text{ and } \frac{2(\alpha-1)}{\alpha^2} <q \le \frac12,\\
(\log x)^{q\alpha^2/2} (\log\log x)^{O(1)} & \text{ if } \alpha\ge 2 \text{ and }0 <q < \frac12.\\ \end{cases}
\end{equation}
We note that the case $q=1/2$ and $\alpha \geq 1$ readily follows from the work of 
Bondarenko, Heap and Seip \cite{BHS} who showed that 
$$
\Psi_{1, d_{\alpha}}(x) = (\log x)^{\alpha^2/4} (\log\log x)^{O(1)}.
$$

In this paper, we determine the order of magnitude of the pseudomoments $\Psi_{2q, d_{\alpha}}(x)$ in the remaining range $0<q\le 1/2$ and $0<\alpha<1$.  A corollary of our main result shows that the estimate \eqref{PseudoPowerZeta} is valid in this range. \begin{corollary}\label{zeta}
Let $0<q \le 1/2$  and $0< \alpha < 1$ be fixed. Then we have 
$$ \Psi_{2q, d_{\alpha}}(x) \asymp (\log x)^{(q \alpha)^2}.$$
\end{corollary}

More generally, we consider Euler products 

$$ 
\mathcal{G}(s)= \sum_{n\geq 1} \frac{g(n)}{n^s} = \prod_{p} \left(1+\frac{g(p)}{p^s}+ \frac{g(p^2)}{p^{2s}}+\cdots\right),
$$ 
for $\re(s)>1$, where $g :\mathbb{N} \to \mathbb{C}$ is a multiplicative function satisfying
\begin{equation}\label{gBoundCond}
|g(n)| \le \min(A n^{\theta}, B^{\Omega(n)})
\end{equation}
for some constants $A, B > 0$ and fixed $0 < \theta < 
\frac{1}{48}$.  
We also assume that $g$ additionally satisfies
\begin{equation} \label{SelbergOrtho}
\lambda_g(x) := \sum_{p\leq x} \frac{|g(p)|^2}{p}= \alpha \log\log x+ O(1),
\end{equation}
for some $\alpha > 0$. 
This estimate would follow from the Selberg orthonormality conjectures if $\mathcal{G}$ is  an $L$-function from the Selberg class. Note that \eqref{gBoundCond} already implies that $\lambda_g(x) \le B^2 \log \log x + O(1)$.
	
For $q>0$ we define the $2q$-th pseudomoment of $\mathcal{G}$ by
\[ \Psi_{2q,g}(x) := \lim_{T \to \infty} \frac{1}{T} \int_T^{2 T} \left| \sum_{n \le x} \frac{g(n)}{n^{1/2+i t}}\right|^{2 q} \, dt. \]

Bondarenko, Heap and Seip \cite{BHS} showed that if \eqref{gBoundCond} and \eqref{SelbergOrtho} hold then for $q>0$ fixed we have
 $$ \Psi_{2 q, g}(x) \asymp  (\log x)^{q^2\alpha},$$
 if $q>1/2$, 
 $$ (\log x)^{\alpha/4} \ll \Psi_{2 q, g}(x) \ll (\log x)^{\alpha/4}\log\log x,
 $$ 
if $q=1/2$, 
and 
	\begin{equation}\label{BHS}
	(\log x)^{q^2\alpha} \ll \Psi_{2 q, g}(x) \ll (\log x)^{q\alpha}.
	\end{equation}
if $0<q<1/2$. 

Our main result shows that when $0<q\le 1/2$ is fixed, the correct order of magnitude of the pseudomoments $\Psi_{2 q, g}(x)$ corresponds to the lower bound of \eqref{BHS} when  $0<\alpha < 1$.
	
	\begin{theorem}\label{Main}
	Suppose that $g :\mathbb{N} \to \mathbb{C}$ is a multiplicative function satisfying \eqref{gBoundCond}, as well as \eqref{SelbergOrtho} for some (fixed) $0<\alpha < 1$. Then uniformly over $0 < q 
	\le 1/2$, we have 
	\begin{align*} \Psi_{2 q, g}(x) &\ll (\log x)^{q^2 \alpha}  \min \left\{\frac{1}{q},\log \log x \right\} \min \left\{ \frac{1}{q^3}, (\log \log \log x)^2 \right\}. \end{align*}
	In particular, if $0<q\le 1/2$ is fixed then 
	$$ \Psi_{2 q, g}(x) \asymp  (\log x)^{q^2\alpha}.$$
	\end{theorem}

	The reader might want to compare this estimate to \cite[Theorem 5]{AHZ}, where a uniform bound over small $q$ is deduced for $\alpha = 1$. It should be noted that we have the slightly smaller factor $\min \{ \frac{1}{q}, \log \log x \}$ in place of $\min \{ \frac{1}{q^2}, \log \log x \}$ there. This uses in a fundamental way that $\alpha$ is not too close to $1$, although it is possible that this difference could be an artifact of the proof. 
	
	It should also be remarked that it is unclear to the authors if the dependency on $q$ in the above theorem is anywhere close to optimal. Moreover, achieving a bound with uniformity with respect to $\alpha$, while perhaps possible, will likely also create a blow-up around certain points. For example, a certain contribution that arises in the proof would require one to replace the term $\min \left \{ \frac{1}{q}, \log \log x \right \}$ by $\min \left \{ \frac{1}{q^2}, \frac{1}{q(1-\alpha)}, \log \log x \right \}$, uniformly for (say) $\frac{1}{2} \le \alpha \le 1$.
	
	We also remark that the proof gives very slightly better bounds on certain ranges of $q$, but the above is somewhat less convoluted to state.
	
	
	
	Our arguments also yield estimates for $\Psi_{2 q, g}(x)$ when $\alpha \ge 1$ that are of the same shape as in \eqref{Maxim}, thus generalizing the bounds there from $d_{\alpha}$ with $\alpha \ge 1$ to general $g$ satisfying \eqref{gBoundCond} and \eqref{SelbergOrtho}. In particular, they all apply to $d_\alpha$ for any $\alpha \in \mathbb{C}$, so both \eqref{Maxim} and Corollary \ref{zeta} hold by replacing $\alpha$ by $|\alpha|$ in the assumptions on $\alpha$ and the right-hand sides of the respective estimates. We omit the details since the combination of arguments from \cite{Ge} and bounds from this work yield the respective claims in a rather straightforward manner.


\section{Preliminary results}
In this section we collect together several preliminary results that will be useful in our subsequent work.
We begin by recording the following lemma which allows us to bound sums involving the multiplicative function $g$, since $g$ satisfies \eqref{gBoundCond}. This corresponds to \cite[Number Theory Result 1]{Ha2}, which is a consequence of Lemma 2.1 of Lau, Tenenbaum and Wu \cite{LTW}.
\begin{lemma}\label{HLTW}
	Let $0 < \delta < 1$ and $\alpha \ge 1$. Suppose that $\max \{ 3 ,2 \alpha \} \le y \le z \le y^{10}$ and $1 < u \le v(1-y^{- \delta})$. Then we have
	\[ \sum_{\substack{ u \le n \le v \\ p \, | \, n \Rightarrow y < p \le z }} \alpha^{\Omega(n)} \ll_\delta \frac{(v-u) \alpha}{\log y} \prod_{y < p \le z} \left( 1 - \frac{\alpha}{p} \right)^{-1}. \] 
\end{lemma}
We will need the following standard version of Plancherel's Theorem for Dirichlet series (see for example \cite[(5.26)]{MV}), in order to relate the pseudomoments $\Psi_{2q, g}(x)$ to averages of random Euler products. 
\begin{lemma}\label{PlancherelLemma}
	Let $(a_n)_{n \ge 1}$ be a sequence of complex numbers and let $A(s) = \sum_{n \ge 1} \frac{a_n}{n^s}$ denote the associated Dirichlet series with abscissa of convergence $\sigma_c$. Then for any $\sigma > \max \{ 0, \sigma_c \}$, we have
	\begin{equation}\label{Plancherel} \int_0^\infty \frac{ \left| \sum_{n \le x} a_n \right|^2}{x^{1+2 \sigma}} \, dx = \frac{1}{2 \pi} \int_{\mathbb{R}} \left| \frac{A(\sigma + i t)}{\sigma + i t} \right|^2 \, dt. \end{equation}
\end{lemma}
The next result allows us to control short sums of the function $|g(n)|^2$ in terms of its logarithmic mean. This was established by Nair and Tenenbaum (see  \cite[Corollary 3]{NT}), generalizing and refining results of Shiu \cite{Sh}.
\begin{lemma}\label{NairTenenbaum}
	Let $g :\mathbb{N} \to \mathbb{C}$ be a multiplicative function satisfying \eqref{gBoundCond}. Then, uniformly for $y$ in the range $x^{12\theta} \leq y\leq x$ we have 
	$$\sum_{x-y\leq n\leq x} |g(n)|^2 \ll \frac{y}{\log x} \sum_{n\leq x} \frac{|g(n)|^2}{n}.$$
	
	\end{lemma}
	
	Let $X : \mathbb{N} \to \mathbb{C}$ denote a \textit{Steinhaus random multiplicative function}, i.e. a (random) completely multiplicative function whose values $(X(p))_{p \text{ prime}}$ are independent and identically distributed (i.i.d.) according to a uniform distribution on the complex unit circle (a Steinhaus distribution). Our next result provides estimates for the expectations of certain random Euler products, and is a generalization of \cite[Lemma 8]{Ge}, compare also \cite[Euler Product Result 1]{Ha1}. 
	\begin{proposition}\label{MixedRandomMoments}
 Let 
	$g$ be a multiplicative function satisfying \eqref{gBoundCond} for some constants $A, B > 0$ and fixed $0 < \theta < 
\frac{1}{4}$. Suppose that $X$ is a Steinhaus random multiplicative function and define, for $2 \le y < z$,
	\[ F(s) := \prod_{y \leq p \le z} \left(1 + \frac{g(p) X(p)}{p^s} + \frac{g(p^2) X(p)^2}{p^{2 s}} + \dots \right). \]
	Then for any real numbers $\sigma \ge -10/\log z$,  $|a| ,|b|\leq 10$ and any real number $t$ we have
\begin{equation}\label{MixedRandomMoments1}
\begin{aligned}
 & \mathbb{E} \left[\left|F\left(\frac{1}{2} + \sigma \right)\right|^{2a} \left|F\left( \frac{1}{2} + \sigma + i t \right)\right|^{2b} \right] \\
 &= \exp \left(\sum_{y \leq  p \le z}\big(a^2+b^2+2ab\cos(t\log p)\big) \frac{|g(p)|^2}{p^{1 + 2 \sigma}} + O\left(\frac{1}{y^{1-4 \theta} \log y}\right) \right), 
 \end{aligned}
\end{equation}
	where the implicit constant might depend on $A, B,$ and  $\theta$. Moreover, if $|t| \ll 1/\log z$ then 
\begin{equation}\label{MixedRandomMoments2}
 \mathbb{E} \left[\left|F\left( \frac{1}{2} + \sigma \right)\right|^{2a} \left|F\left( \frac{1}{2} + \sigma + i t \right)\right|^{2 b} \right] \ll \exp \left((a+b)^2 \sum_{y \leq p \le z} \frac{|g(p)|^2}{p}\right).
\end{equation}
\end{proposition}


\begin{proof} Let 
$$ F_p(s) = 1 + \sum_{j=1}^{\infty} \frac{g(p^j) X(p)^j}{p^{js}}.$$
Then, by the independence of the $X(p)$'s we have 
$$
\mathbb{E} \left[\left|F\left( \frac{1}{2} + \sigma \right)\right|^{2a} \left|F\left( \frac{1}{2} + \sigma + i t \right)\right|^{2 b} \right]= 
\prod_{y\leq p\leq z} \mathbb{E} \left[\left|F_p\left(\frac12+\sigma\right)\right|^{2a}\left|F_p\left(\frac12+\sigma+it\right)\right|^{2b}\right].
$$
Furthermore, for any $s$ with $\re(s)=1/2+\sigma$, any prime $p\leq z$ and any real number $|c|\leq 10$  we have 
\begin{align*}
& \left|F_p\left(s \right)\right|^{2c}  = \exp\big(2c \re \log F_p(s)\big)\\
&=  \exp\left(2c \re \log \left(1 + \frac{g(p) X(p)}{p^s} + \frac{g(p^2) X(p)^2}{p^{2 s}} + \frac{g(p^3) X(p)^3}{p^{3 s}} + O\left(\frac{1}{p^{2-4\theta}}\right)\right)\right)\\
&=  \exp\left(2c \re \left(\frac{g(p) X(p)}{p^s} + \frac{2g(p^2) X(p)^2-g(p)^2X(p)^2}{2p^{2 s}} + \frac{3g(p^3) X(p)^3+ g(p)^3X(p)^3}{3p^{3 s}} + O\left(\frac{1}{p^{2-4\theta}}\right)\right)\right),\\
\end{align*}
since $p^{-\sigma}\leq p^{10/\log z} \ll 1$, which follows from the fact that $p\leq z$. 
We shall use these estimates for $\left|F_p\left( 1/2 + \sigma \right)\right|^{2a}$ and  $\left|F_p\left(1/2+ \sigma + i t \right)\right|^{2 b}$ and apply the series expansion of the exponential to compute the expectation $$\mathbb{E} \left[\left|F_p\left( \frac{1}{2} + \sigma \right)\right|^{2a} \left|F_p\left( \frac{1}{2} + \sigma + i t \right)\right|^{2 b} \right].$$
To this end, we shall use the following easy observations to simplify the computations. First we note that for any complex number $w$, and positive integer $n$ one has 
$\ex[\re(w X(p)^n)]=0$. Moreover, by translation-invariance law $w X(p)$ has the same distribution as $|w| X(p)$, and hence $\ex[\re(w X(p))^3]=0$. One can also verify that $\ex[\re(w X(p))^2\re(X(p))]=0$. 
Using these facts, an easy computation  shows that for $p\leq z$
\begin{align*}
& \mathbb{E} \left[\left|F_p\left( \frac{1}{2} + \sigma \right)\right|^{2a} \left|F_p\left( \frac{1}{2} + \sigma + i t \right)\right|^{2 b} \right]\\
&= 1+ \frac{2}{p^{1+2\sigma}}\ex\left[\Big(a\re\big(g(p) X(p)\big)+ b \re\big(g(p) X(p) p^{-it}\big)\Big)^2\right] + O\left(\frac{1}{p^{2-4\theta}}\right).
\end{align*}
Finally, using that $\ex[(\re(w X(p)))^2]= |w|^2\ex[(\re X(p))^2]= |w|^2/2$  and $\ex[(\re X(p))(\re(w X(p))]=|w| \cos(\arg w)/2,$ we deduce that  
\begin{align*}
\mathbb{E} \left[\left|F_p\left( \frac{1}{2} + \sigma \right)\right|^{2a} \left|F_p\left( \frac{1}{2} + \sigma + i t \right)\right|^{2 b} \right]
&= 1+ (a^2+b^2+2ab \cos(t\log p))\frac{|g(p)|^2}{p^{1+2\sigma}}+ O\left(\frac{1}{p^{2-4\theta}}\right)\\
&= \exp\left((a^2+b^2+2ab \cos(t\log p))\frac{|g(p)|^2}{p^{1+2\sigma}}+ O\left(\frac{1}{p^{2-4\theta}}\right)\right).
\end{align*}
Taking the product of this estimate over the primes $y\leq p\leq z$ implies \eqref{MixedRandomMoments1}.

Now, to obtain \eqref{MixedRandomMoments2}, we use that $|g(p)|\leq B$, $|t| \ll 1/\log z$ and $p^{-\sigma}\ll 1$ to obtain that 
$$
\sum_{y\leq p\leq z}\frac{|g(p)|^2\cos(t\log p)}{p^{1+2\sigma}}= \sum_{y\leq p\leq z} \frac{|g(p)|^2}{p^{1+2\sigma}} +O\left(|t|^2\sum_{y\leq p\leq z}\frac{(\log p)^2}{p}\right)= \sum_{y\leq p\leq z} \frac{|g(p)|^2}{p^{1+2\sigma}} + O(1). 
$$
Thus, if we put $\sigma_0= -10/\log z$ we deduce that 
$$
\sum_{y\leq p \leq z} \big(a^2+b^2+2ab \cos(t\log p)\big)\frac{|g(p)|^2}{p^{1+2\sigma}}
= (a+b)^2\sum_{y\leq p \leq z} \frac{|g(p)|^2}{p^{1+2\sigma}} +O(1) \leq (a+b)^2\sum_{y\leq p \leq z} \frac{|g(p)|^2}{p^{1+2\sigma_0}} +O(1).
$$
The result follows upon noting that 
$$
\sum_{y\leq p \leq z} \frac{|g(p)|^2}{p^{1+2\sigma_0}} +O(1)=\sum_{y\leq p \leq z} \frac{|g(p)|^2}{p}  +O\left(\sum_{y\leq p \leq z} \frac{|\sigma_0|\log p}{p}+1\right)= \sum_{y\leq p \leq z} \frac{|g(p)|^2}{p}  +O\left(1\right).
$$
This completes the proof.

\end{proof}
	
	\section{From pseudomoments to random Euler product integrals}
	
 In this section, we establish an upper bound for the pseudomoments $\Psi_{2q,g}(x)$ in terms of certain moments of integrals of random Euler products. We let $X$ be a Steinhaus random multiplicative function and define  for $\re(s) > \theta$ 
\begin{equation}\label{DefGk}
 G_k(s) := \prod_{p \le x^{e^{-k}}} \left( 1 + \frac{X(p) g(p)}{p^s} + \frac{X(p)^2 g(p^2)}{p^{2 s}} + \dots \right) 
\end{equation}
to be the (truncated) Euler product associated to $X g$  over primes up to $x^{e^{-k}}$. We also let $P(n)$ denote the largest prime factor of $n$.
\begin{proposition}\label{REP}
		Let $g : \mathbb{N} \to \mathbb{C}$ be a multiplicative function satisfying \eqref{gBoundCond}. Let $x$ be sufficiently large, and put $K = [\log \log \log x]$. 
		Then, uniformly over $x$ and $0 
		< q \le 1/2$, we have
		\begin{align*}
		\Psi_{2q,g}(x) \ll \frac{1}{(\log x)^q} &\sum_{0 \le k \le K} \mathbb{E} \Bigg[ \bigg( \int_1^{x^{1-e^{-(k+1)}}} \bigg| \sum_{\substack{n > z \\ P(n)  \le x^{e^{-(k+1)}} }} \frac{X(n) g(n) }{\sqrt{n}}  \bigg|^2 \, \frac{dz}{z^{1-2 k / \log x}} \bigg)^q \bigg] \\ + &\sum_{0 \le k \le K+1} e^{- e^k q} \mathbb{E} \left[ |G_k(1/2)|^{2 q} \right] + K \exp 
		\left(-(1+o(1))q \sqrt{\log x}\right) + 1.
		\end{align*}
	\end{proposition}
We note that the proof of this proposition is rather similar to the one of \cite[Proposition 2.3]{Ge} combined with the modifications in \cite[Theorem 5]{AHZ} in order to achieve uniformity over $q$. We will nonetheless include the proof for the sake of completeness.
\begin{proof}
The Bohr correspondence (compare e.g. \cite[Section 3]{SS}) implies that
	\[ \Psi_{2 q, g}(x) = \Big\Vert \sum_{n \le x} \frac{g(n) X(n)}{\sqrt{n}} \Big\Vert_{2 q}^{2q}, \]
	where $\Vert \cdot \Vert_{2 q} = \mathbb{E}[ |\cdot|^{2 q}]^{1/{2 q}}$. Note that this does not define a norm when $q < 1/2$, but only a pseudonorm (but we might still sometimes refer to it as a norm). 
	
	Denoting $G(s):=G_0(s)$, observe that 
	\begin{align*}
	\Big\Vert \sum_{n \le x} \frac{X(n) g(n)}{\sqrt{n}} \Big\Vert_{2 q}^{2q}
	&\le \Big\Vert \sum_{P(n) \le x} \frac{X(n) g(n)}{\sqrt{n}} \Big\Vert_{2 q}^{2q} + \Big\Vert \sum_{\substack{n > x \\ P(n) \le x}} \frac{X(n) g(n)}{\sqrt{n}} \Big\Vert_{2 q}^{2q} \\
	&= \Big \Vert \prod_{p \le x} \left( 1 + \frac{X(p) g(p) }{p^{1/2}} + \frac{X(p)^2 g(p^2)}{p} + \dots \right)\Big \Vert_{2 q}^{2 q} + \Big\Vert \sum_{\substack{n > x \\ P(n) \le x}} \frac{X(n) g(n)}{\sqrt{n}} \Big\Vert_{2 q}^{2q}  \\
	&= \mathbb{E} \left[ |G(1/2)|^{2 q} \right] + \Big\Vert \sum_{\substack{n > x \\ P(n) \le x}} \frac{X(n) g(n)}{\sqrt{n}} \Big\Vert_{2 q}^{2q}.
	\end{align*}
	We can then subdivide the sum according to the size of the largest prime factor, to obtain that
	\begin{align*} \Big\Vert \sum_{n \le x} \frac{X(n) g(n)}{\sqrt{n}} \Big\Vert_{2 q}^{2q} \le \mathbb{E} \left[ |G(1/2)|^{2 q} \right] + \sum_{0 \le k \le K} &\Big\Vert \sum_{\substack{n > x \\ x^{e^{-(k+1)} < P(n) \le x^{e^{-k}} } }} \frac{X(n) g(n)}{\sqrt{n}} \Big\Vert_{2 q}^{2 q} \\ + &\Big \Vert \sum_{\substack{n > x \\ P(n) \le x^{e^{-(K+1)}}  }} \frac{X(n) g(n)}{\sqrt{n}} \Big \Vert_{2 q}^{2 q}. 
	\end{align*}
	In order to bound the last term, we can trivially bound the $2q$-norm by the $2$-norm and use orthogonality to deduce that 
	\[ \Big \Vert \sum_{\substack{n > x \\ P(n) \le x^{e^{-(K+1)}}}} \frac{X(n) g(n)}{\sqrt{n}} \Big \Vert_{2 q}^{2 q} \le \bigg(\sum_{\substack{n > x \\ P(n) \le x^{e^{-(K+1)}}}} \frac{|g(n)|^2}{n}\bigg)^q.  \]
	But this can be dealt with by means of Rankin's trick: For any constant $C>0$ and any sufficiently large $y$ (depending only on $C$), we have
	\begin{align*}
	\sum_{\substack{n > x \\ P(n) \le y  }} \frac{|g(n)|^2}{n} &\le x^{-C/\log y} \sum_{P(n) \le y } \frac{|g(n)|^2}{n^{1-C/\log y}} = x^{-C/\log y} \prod_{p \le y} \left(1 + \frac{|g(p)|^2}{p^{1-C/\log y}} + \frac{|g(p^2)|^2}{p^{2(1-C/\log y)}} + \dots \right) \\
	&\ll x^{-C/\log y} \prod_{p \le y} \left( 1 + \frac{B^2}{p^{1-C/\log y}} + O \left( p^{-3/2} \right) \right) \\ &\ll x^{-C/\log y} \exp \left( B^2 \sum_{p \le y} \frac{1}{p^{1-C/\log y}} \right) \ll x^{-C/\log y} (\log y)^{B^2}. 
	\end{align*}
	
	Taking $y = x^{1/\log \log x}$ and $C=B^2$ and using that $K = [\log \log \log x]$ thus gives
	\[ \sum_{\substack{n > x \\ P(n) \le x^{e^{-(K+1)}}  }} \frac{|g(n)|^2}{n} \ll (\log x)^{-B^2} \left( \frac{\log x}{\log \log x} \right)^{B^2} \ll 1. \]
	Putting the bounds up to this point together tells us that
	\begin{equation} \Psi_{2q,g}(x) \ll \sum_{0 \le k \le K} \Big\Vert \sum_{\substack{n > x \\ x^{e^{-(k+1)} < P(n) < x^{e^{-k}} } }} \frac{X(n) g(n)}{\sqrt{n}} \Big\Vert_{2 q}^{2 q} +  \mathbb{E} \left[ |G(1/2)|^{2 q} \right] + 1. \end{equation}
	
	Next, let $\mathbb{E}^{(k)}$ denote the conditional expectation given $(X(p))_{p \le x^{e^{-(k+1)}}}$. Using H\"{o}lder's inequality for conditional expectations as well as the independence of $X(p)$ at different primes and orthogonality (compare \cite[Proposition 1]{Ha2}), we have
	\begin{align*}
	&\quad \; \bigg\Vert \sum_{\substack{n > x \\ x^{e^{-(k+1)} < P(n) \le x^{e^{-k}} } }} \frac{X(n) g(n)}{\sqrt{n}} \bigg\Vert_{2 q}^{2 q} \\
	&= \bigg \Vert \sum_{\substack{ m > 1 \\ p \, | \, m \Rightarrow x^{e^{-(k+1)}} < p \le x^{e^{-k}}}} \frac{X(m) g(m)}{\sqrt{m}} \sum_{\substack{ n > x/m \\ P(n)\le x^{e^{-(k+1)}} }} \frac{X(n) g(n)}{\sqrt{n}} \bigg \Vert_{2 q}^{2 q} \\
	&= \mathbb{E} \Bigg[ \mathbb{E}^{(k)} \Bigg[ \; \bigg| \sum_{\substack{m > 1 \\ p \, | \, m \Rightarrow x^{e^{-(k+1)}} < p \le x^{e^{-k}} }} \frac{X(m) g(m)}{\sqrt{m}} \sum_{\substack{ n > x/m \\ P(n)\le x^{e^{-(k+1)}} }} \frac{X(n) g(n)}{\sqrt{n}} \bigg|^{2 q} \; \Bigg] \Bigg] \\
	&\le \mathbb{E} \Bigg[ \Bigg( \mathbb{E}^{(k)} \Bigg[ \bigg|\sum_{\substack{ m > 1 \\ p \, | \, m \Rightarrow x^{e^{-(k+1)}} < p \le x^{e^{-k}} }} \frac{X(m) g(m)}{\sqrt{m}} \sum_{\substack{ n > x/m \\ P(n)\le x^{e^{-(k+1)}} }} \frac{X(n) g(n)}{\sqrt{n}}  \bigg|^2 \Bigg] \Bigg)^q \Bigg] \\
	&= \bigg \Vert \sum_{\substack{m > 1 \\ p \, | \, m \Rightarrow x^{e^{-(k+1)}} < p \le x^{e^{-k}} }} \frac{|g(m)|^2}{m} \; \bigg| \sum_{\substack{ n > x/m \\ P(n)\le x^{e^{-(k+1)}} }} \frac{X(n) g(n)}{\sqrt{n}} \bigg|^2 \, \bigg \Vert_q^q.
	\end{align*}
	
	The next step is to smoothen the inner sum. Again we proceed in a very similar fashion to \cite[Proposition 1]{Ha2}, setting $Y=\exp(\sqrt{\log x})$ (say) and noting that
	\begin{align}
	&\bigg \Vert \sum_{\substack{m > 1 \\ p \, | \, m \Rightarrow x^{e^{-(k+1)}} < p \le x^{e^{-k}} }} \frac{|g(m)|^2}{m} \; \bigg| \sum_{\substack{ n > x/m \\ P(n)\le x^{e^{-(k+1)}} }} \frac{X(n) g(n) }{\sqrt{n}} \bigg|^2 \, \bigg \Vert_q^q \nonumber \\
	\ll \, &\bigg \Vert \sum_{\substack{m > 1 \\ p \, | \, m \Rightarrow x^{e^{-(k+1)}} < p \le x^{e^{-k}} }} \frac{Y |g(m)|^2}{m^2} \int_m^{m(1+1/Y)} \bigg| \sum_{\substack{ n > x/t \\ P(n)\le x^{e^{-(k+1)}} }} \frac{X(n) g(n)}{\sqrt{n}} \bigg|^2 \, dt \, \bigg \Vert_q^q \nonumber \\
	+ \, &\bigg \Vert \sum_{\substack{m > 1 \\ p \, | \, m \Rightarrow x^{e^{-(k+1)}} < p \le x^{e^{-k}} }} \frac{Y |g(m)|^2}{m^2} \int_m^{m(1+1/Y)} \bigg| \sum_{\substack{ x/t < n \le x/m \\ P(n)\le x^{e^{-(k+1)}} }} \frac{X(n) g(n)}{\sqrt{n}} \bigg|^2 \, dt \, \bigg \Vert_q^q. \label{SmoothingError}
	\end{align}
	The range of summation for the inner sum in the second term is rather small, so we might expect this to only give a minor contribution. Indeed, trivially bounding the $q$-norm by the $1$-norm, pulling the expectation inside and then using orthogonality, the second term in (\ref{SmoothingError}) is
	\begin{align*}
	&\le \Bigg( \sum_{\substack{m > 1 \\ p \, | \, m \Rightarrow x^{e^{-(k+1)}} < p \le x^{e^{-k}} }} \frac{Y |g(m)|^2}{m^2} \int_m^{m(1+1/Y)} \mathbb{E} \bigg[ \; \bigg| \sum_{\substack{ x/t < n \le x/m \\ P(n)\le x^{e^{-(k+1)}} }} \frac{X(n) g(n)}{\sqrt{n}} \bigg|^2 \, \bigg] \, dt  \Bigg)^q \\
	&\le \Bigg( \sum_{\substack{ m > 1\\ p \, | \, m \Rightarrow x^{e^{-(k+1)}} < p \le x^{e^{-k}} }} \frac{|g(m)|^2}{m} \sum_{\substack{ \frac{x}{m(1+1/Y)} < n \le \frac{x}{m} \\ P(n)\le x^{e^{-(k+1)}} }} \frac{|g(n)|^2}{n} \Bigg)^q \\
	&\ll \Bigg( \frac{1}{x} \sum_{\substack{m > 1 \\ p \, | \, m \Rightarrow x^{e^{-(k+1)}} < p \le x^{e^{-k}} }} |g(m)|^2 \sum_{\substack{ \frac{x}{m(1+1/Y)} < n \le \frac{x}{m} \\ P(n)\le x^{e^{-(k+1)}} }} |g(n)|^2 \Bigg)^q.
	\end{align*}
	In order to bound this, note first that for any $z \ge 2 B^2$ we have
	\begin{align}
	   \sum_{n \le z} \frac{|g(n)|^2}{n} \le \sum_{P(n)\le z} \frac{|g(n)|^2}{n} &\le \prod_{p \le 2 B^2} \left( 1 + \frac{A p^{2 \theta}}{p} + \frac{A p^{4 \theta}}{p^2} + \dots \right) \prod_{2 B^2 < p \le z} \left( 1 - \frac{B^2}{p} \right)^{-1}
	   \ll (\log z)^{B^2}. \label{TrivialBoundOnInt}
	\end{align}
	We will now make use of Lemma \ref{NairTenenbaum},
	invoking a hyperbola-type argument and subdividing the first sum into the range $1 < m \le \sqrt{x}$ and $m >  \sqrt{x}$. We then interchange the sum on the latter range, and thus have (assuming without loss of generality that $B \ge 1$)
	\begin{align*}
	&\frac{1}{x} \sum_{\substack{m > 1 \\ p \, | \, m \Rightarrow x^{e^{-(k+1)}} < p \le x^{e^{-k}} }} |g(m)|^2 \sum_{\substack{ \frac{x}{m(1+1/Y)} < n \le \frac{x}{m}\\ P(n) \le x^{e^{-(k+1)}}}} |g(n)|^2 \\
	\le &\frac{1}{x} \sum_{\substack{1 < m \le \sqrt{x} \\ p \, | \, m \Rightarrow x^{e^{-(k+1)}} < p \le x^{e^{-k}} }} |g(m)|^2 \sum_{\frac{x}{m(1+1/Y)} < n \le \frac{x}{m}} |g(n)|^2 
	\\ + &\frac{1}{x} \sum_{\substack{n \le\sqrt{x} \\ P(n) \le x^{e^{-(k+1)}} }} |g(n)|^2 \sum_{\substack{ \frac{x}{n(1+1/Y)} < m \le \frac{x}{n} \\ p \, | \, m \Rightarrow x^{e^{-(k+1)}} < p \le x^{e^{-k}} }} |g(m)|^2 \\
	\ll &  \Bigg(\frac{(\log x)^{B^2 - 1}}{Y}  + \frac{1}{Y \log x} \sum_{\substack{n \le \sqrt{x} \\ P(n) \le x^{e^{-(k+1)}} }}  \frac{|g(n)|^2}{n}\Bigg) \sum_{\substack{m > 1 \\ p \, | \, m \Rightarrow x^{e^{-(k+1)}} < p \le x^{e^{-k}} }} \frac{|g(m)|^2}{m} \\
	\ll &  \Bigg(\frac{(\log x)^{B^2 - 1}}{Y}  + \frac{1}{Y \log x} \prod_{p \le x^{e^{-(k+1)}}} \left( 1 + \frac{
	|g(p)|^2}{p} \right)\Bigg) \prod_{x^{e^{-(k+1)}} < p \le x^{e^{-k}}} \left( 1 + \frac{|g(p)|^2}{p}  \right) \\ 	
	\ll &\frac{(\log x)^{B^2-1}}{Y}.
	\end{align*} 
	Taking $q$-th powers and summing over $0 \le k \le K$, we arrive at a contribution
	\[ \ll K \exp \left( -(1+o(1)) q \sqrt{\log x } \right). \]

	Regarding the first term in (\ref{SmoothingError}), we can interchange sum and integral to arrive at
	\begin{equation}\label{tIntegral} \bigg \Vert \int_{x^{e^{-(k+1)}}}^\infty \bigg| \sum_{\substack{n > x/t \\ P(n) \le x^{e^{-(k+1)}} }} \frac{X(n) g(n)}{\sqrt{n}}  \bigg|^2 \sum_{\substack{t/(1+1/Y) < m \le t \\ p \, | \, m \Rightarrow x^{e^{-(k+1)}} < p \le x^{e^{-k}} }} \frac{Y |g(m)|^2 }{m^2} \, dt \bigg \Vert_q^q. \end{equation}
	
	Concluding our estimates so far, we have now proven that
	\begin{align*}
	\Psi_{2 q, g}(x) &\ll \sum_{0 \le k \le K} \bigg \Vert \int_{x^{e^{-(k+1)}}}^\infty \bigg| \sum_{\substack{n > x/t \\ P(n) \le x^{e^{-(k+1)}} }} \frac{X(n) g(n)}{\sqrt{n}}  \bigg|^2 \sum_{\substack{t/(1+1/Y) < m \le t \\ p \, | \, m \Rightarrow x^{e^{-(k+1)}} < p \le x^{e^{-k}} }} \frac{Y |g(m)|^2 }{m^2} \, dt \bigg \Vert_q^q \\ &+ \mathbb{E} \left[ |G(1/2)|^{2 q} \right] + K \exp \left( -(1+o(1)) q \sqrt{\log x} \right) + 1.
	\end{align*}
	
	For the inner sum in (\ref{tIntegral}), note that $|g(m)|^2 \le B^{2 \Omega(m)}$, and that $m > t/(1+1/Y)$ and $p \, | \, m \Rightarrow x^{e^{-(k+1)}} < p \le x^{e^{-k}}$ imply that $m$ has $\ge \frac{e^k \log (t/2)}{\log x}$ prime divisors (counted with multiplicity). Thus, Lemma \ref{HLTW} implies
	\begin{equation}\label{SumRestrictedg1}
	\begin{aligned}
	\sum_{\substack{t/(1+1/Y) < m \le t \\ p \, | \, m \Rightarrow x^{e^{-(k+1)}} < p \le x^{e^{-k}} }} \frac{Y |g(m)|^2 }{m^2} & \ll \frac{Y}{t^2} \sum_{\substack{t/(1+1/Y) < m \le t \\ p \, | \, m \Rightarrow x^{e^{-(k+1)}} < p \le x^{e^{-k}} }} B^{2 \Omega(m)} \\
	&\le \frac{Y}{t^2} e^{-\frac{e^k \log (t/2)}{\log x}} \sum_{\substack{t/(1+1/Y) < m \le t \\ p \, | \, m \Rightarrow x^{e^{-(k+1)}} < p \le x^{e^{-k}} }} (e B^2)^{\Omega(m)} \\
	&\ll \frac{e^k e^{- \frac{e^k \log (t/2)}{\log x}}}{t \log x}.
	\end{aligned}
		\end{equation}
		In particular, one has
	\begin{equation}\label{SumRestrictedg2}
	\sum_{\substack{t/(1+1/Y) < m \le t \\ p \, | \, m \Rightarrow x^{e^{-(k+1)}} < p \le x^{e^{-k}} }} \frac{Y |g(m)|^2 }{m^2}\ll \frac{1}{t \log t}.
	\end{equation}
	Subdividing the range of integration in (\ref{tIntegral}) into $t \le x$ and $t > x$,  and using \eqref{SumRestrictedg1} in the latter and \eqref{SumRestrictedg2} in the former range,   we thus upper-bound (\ref{tIntegral}) by
	
	\begin{align*}
	&\quad \; \bigg \Vert \int_{x^{e^{-(k+1)}}}^x \bigg| \sum_{\substack{n > x/t \\ P(n) \le  x^{e^{-(k+1)}} }} \frac{X(n) g(n)}{\sqrt{n}}  \bigg|^2 \, \frac{dt}{t \log t} \bigg \Vert_q^q \\
	&+ \frac{e^{k q}}{(\log x)^q} \bigg \Vert \sum_{\substack{n \ge 1 \\ P(n) \le x^{e^{-(k+1)}} }} \frac{X(n) g(n)}{\sqrt{n}}  \bigg \Vert_{2 q}^{2 q} \left( \int_x^\infty \frac{dt}{t^{1+\frac{e^k }{\log x}}} \right)^q \\
	&\ll \bigg \Vert \int_{x^{e^{-(k+1)}}}^x \bigg| \sum_{\substack{n > x/t \\ P(n) \le x^{e^{-(k+1)}}}} \frac{X(n) g(n)}{\sqrt{n}}  \bigg|^2 \, \frac{dt}{t \log t} \bigg \Vert_q^q
	+ e^{-e^k q} \mathbb{E}[|G_{k+1} (1/2)|^{2 q}].
	\end{align*}
	Substituting $z=x/t$, the first term equates to
	\begin{align*} &\bigg \Vert \int_1^{x^{1-e^{-(k+1)}}} \bigg| \sum_{\substack{n > z \\P(n) \le x^{e^{-(k+1)}}}} \frac{X(n) g(n)}{\sqrt{n}}  \bigg|^2 \, \frac{dz}{z \log(x/z)} \bigg \Vert_q^q \\
	\ll &\frac{1}{(\log x)^q} \bigg \Vert \int_1^{x^{1-e^{-(k+1)}}} \bigg| \sum_{\substack{n > z \\ P(n) \le x^{e^{-(k+1)}}}} \frac{X(n) g(n)}{\sqrt{n}}  \bigg|^2 \, \frac{dz}{z^{1-2 k / \log x}} \bigg \Vert_q^q,
	\end{align*}
	using that $\log(x/z) \gg z^{- 2 k / \log x} \log x$,  which follows  from the simple estimates $\log(x/z)\gg \log x$ if $z\leq \sqrt{x}$, and $\log(x/z)\gg e^{-k}\log x$ if for $\sqrt{x}\leq z\leq x^{1-e^{-(k+1)}}$. Putting everything together gives the claim.
\end{proof}

It should be noted that several of the steps in this argument do not require $|g(p)|$ to be bounded, which might be viewed as our strongest assumption on $g$. However, this appears to be crucial for our method in order to be able to apply Lemma \ref{NairTenenbaum}.

The next step of the argument is the application of Lemma \ref{PlancherelLemma} (Plancherel's Theorem for Dirichlet series), which proceeds in a similar fashion as in \cite{Ge}. We can not apply it directly to the first term in Proposition \ref{REP} because we would have $\sigma < 0$. Thus, we first need to apply partial summation and Cauchy-Schwarz to increase the exponent of $z$ to be slightly bigger than $1$, at the expense of making the inner sum larger to an extent that turns out not to matter. Moreover, partial summation allows us to switch back from an inner sum over $n > z$ to $n \le z$, which is what we need in order to apply Lemma \ref{PlancherelLemma}.

\begin{proposition}\label{REP2}
	Let $g, K$ and $q$ be as in Proposition \ref{REP}. Then uniformly over $0 \le k \le K$, we have
	\begin{align*}
	&\mathbb{E} \Bigg[ \bigg( \int_1^{x^{1-e^{-(k+1)}}} \bigg| \sum_{\substack{n > z \\ P(n) \le x^{e^{-(k+1)}}}} \frac{X(n) g(n)}{\sqrt{n}}  \bigg|^2 \, \frac{dz}{z^{1-2 k / \log x}} \bigg)^q \bigg]
	\ll \mathbb{E} \Bigg[ \bigg( \int_{\mathbb{R}} \frac{\left| G_{k+1} \left( \frac{1}{2} - \frac{2 (k+1)}{\log x} + it \right) \right|^2}{\left| \frac{2 (k+1)}{\log x} + i t \right|^2} \, dt \bigg)^q \Bigg].
	\end{align*}
\end{proposition}

\begin{proof}
	Firstly, we note that
	\begin{align}
	\mathbb{E} &\Bigg[ \bigg( \int_1^{x^{1-e^{-(k+1)}}} \bigg| \sum_{\substack{n > z \\ P(n) \le x^{e^{-(k+1)}}}} \frac{X(n) g(n)}{\sqrt{n}}  \bigg|^2 \, \frac{dz}{z^{1-2 k / \log x}} \bigg)^q \bigg] \nonumber \\
	= \lim_{y \to \infty} &\bigg \Vert \int_1^{x^{1-e^{-(k+1)}}} \bigg| \sum_{\substack{z < n \le y \\ P(n) \le x^{e^{-(k+1)}}}} \frac{X(n) g(n)}{\sqrt{n}}  \bigg|^2 \, \frac{dz}{z^{1-2 k / \log x}} \bigg \Vert_q^q. \label{limsup}
	\end{align}
	Partial summation applied to the inner sum implies that for any $y > z$ and $\sigma > 0$ we have
	\begin{align}
	&\bigg| \sum_{\substack{z < n \le y \\ P(n) \le x^{e^{-(k+1)}}}} \frac{X(n) g(n)}{\sqrt{n}}  \bigg|^2 =  \bigg| \sum_{\substack{z < n \le y \\ P(n) \le x^{e^{-(k+1)}}}} n^{-\sigma} \frac{X(n) g(n)}{n^{1/2-\sigma}}  \bigg|^2 \nonumber \\
	\le &\bigg| y^{-\sigma} \sum_{\substack{n \le y \\ P(n) \le x^{e^{-(k+1)}}}} \frac{X(n) g(n)}{n^{1/2-\sigma}} \bigg|^2 + \bigg| z^{-\sigma} \sum_{\substack{n \le z \\ P(n) \le x^{e^{-(k+1)}}}} \frac{X(n) g(n)}{n^{1/2-\sigma}} \bigg|^2 \nonumber \\
	+ &\bigg| \sigma \int_z^y \sum_{\substack{n \le u \\ P(n) \le x^{e^{-(k+1)}}}} \frac{X(n) g(n)}{n^{1/2-\sigma}} \, \frac{du}{u^{1+\sigma}} \bigg|^2. \label{PartialSum}
	\end{align}
	Plugging the first term of (\ref{PartialSum}) into (\ref{limsup}), using that the inner sum does not depend on $z$ and trivially bounding the arising $2 q$-norm by the $2$-norm, gives a contribution
	\begin{align*}
	\le \Big( \int_1^{x^{1-e^{-(k+1)}}} \, \frac{dz}{z^{1-2k/\log x}} \Big)^q \limsup_{y \to \infty} \bigg( y^{-2 \sigma} \sum_{\substack{n \le y \\ P(n) \le x^{e^{-(k+1)}} }} \frac{|g(n)|^2}{n^{1-2 \sigma}} \bigg)^q.
	\end{align*}
	But
	\[ \sum_{\substack{n \le y \\ P(n) \le x^{e^{-(k+1)}} }} \frac{|g(n)|^2}{n^{1-2 \sigma}} \le \sum_{P(n) \le x^{e^{-(k+1)}}} \frac{|g(n)|^2}{n^{1-2 \sigma}} \]
	is bounded independently of $y$ (since $0 < \sigma < \frac{1}{8}$ will not depend on $y$), hence the contribution vanishes in the limit.
	
	If we plug in the third term of (\ref{PartialSum}) into (\ref{limsup}) with $y$ fixed for now, we arrive at a contribution 
	\[ \bigg \Vert \sigma^2 \int_1^{x^{1-e^{-(k+1)}}} \bigg| \int_z^y \sum_{\substack{n \le u \\ P(n) \le x^{e^{-(k+1)}}}} \frac{X(n) g(n)}{n^{1/2-\sigma}} \, \frac{du}{u^{1+\sigma}} \bigg|^2 \, \frac{dz}{z^{1-2 k / \log x}} \bigg \Vert_q^q. \]
	Applying Cauchy-Schwarz to the inner integral and then extending the arising (non-negative) integrals to $\infty$, we see that the last expression is
	\begin{align*} &\le \bigg \Vert \sigma^2 \int_1^{x^{1-e^{-(k+1)}}} \bigg( \int_z^y \Big| \sum_{\substack{n \le u \\ P(n) \le x^{e^{-(k+1)}}}} \frac{X(n)g(n)}{n^{1/2-\sigma}} \Big|^2 \, \frac{du}{u^{1+\sigma}} \bigg) \bigg( \int_z^y \, \frac{du}{u^{1+\sigma}} \bigg) \, \frac{dz}{z^{1-2 k / \log x}} \bigg \Vert_q^q \\
	&\le \bigg \Vert \sigma \int_1^{x^{1-e^{-(k+1)}}}  \int_z^\infty \Big| \sum_{\substack{n \le u \\ P(n) \le x^{e^{-(k+1)}}}} \frac{X(n)g(n)}{n^{1/2-\sigma}} \Big|^2 \, \frac{du}{u^{1+\sigma}} \, \frac{dz}{z^{1-2 k / \log x + \sigma}} \bigg \Vert_q^q.
	\end{align*}
	Note that the last expression is independent of $y$, so we may take the limit. Lastly, interchanging the two integrals and taking $\sigma = \frac{4 (k+1)}{\log x}$, we see that this is
	\begin{equation}\label{BeforePl} \ll \bigg \Vert \int_1^\infty \Big| \sum_{\substack{n \le u \\ P(n) \le x^{e^{-(k+1)}}}} \frac{X(n)g(n)}{n^{1/2-\sigma}} \Big|^2 \, \frac{du}{u^{1+\sigma}} \bigg \Vert_q^q. \end{equation}
	
	The second term of (\ref{PartialSum}), which is independent of $y$, gives a contribution in (\ref{limsup}) of
	\[ \bigg \Vert \int_1^{x^{1-e^{-(k+1)}}} \Big| \sum_{\substack{n \le z \\ P(n) \le x^{e^{-(k+1)}}}} \frac{X(n)g(n)}{n^{1/2-\sigma}} \Big|^2 \, \frac{dz}{z^{1- 2 k / \log x + 2 \sigma}} \bigg \Vert_q^q \]
	and is thus absorbed into (\ref{BeforePl}). But now we can finally apply Lemma \ref{PlancherelLemma} to (\ref{BeforePl}), and we obtain that it is
	\[ \ll \mathbb{E} \Bigg[ \bigg( \int_{\mathbb{R}} \frac{\left| G_k \left( \frac{1}{2} - \frac{\sigma}{2} + it \right) \right|^2}{\left| \frac{\sigma}{2} + i t \right|^2} \, dt \bigg)^q \Bigg],  \]
	from which the claim follows.
	
\end{proof}


\section{Upper bounds for low pseudomoments: Proof of Theorem \ref{Main}}

 Let $x$ be a sufficiently large real number and put $K:=[\log\log\log x].$ To shorten our notation, we let 
$$ H_{k, \sigma}(t):= G_k\left(\frac{1}{2}+\sigma+it\right),$$
where $k$ is a non-negative integer and $G_k$ is defined in \eqref{DefGk}.

In this section, we shall complete the proof of Theorem \ref{Main}. To this end, we will bound the $q$-th moment of $\int_T^{2T} |H_{k, \sigma}(t)|^2 dt$ for different ranges of $T$ with respect to $k$ and $x$.
\begin{proposition}\label{SmallT}
Let $0<q\le 1/2$ and $0\leq k\leq K+1$. Assume further that $0<T\leq e^k/\log x$ and that $\sigma\geq -2(k+1)/\log x$. Then we have 
$$ \ex \left[ \left(\int_T^{2T} |H_{k, \sigma}(t)|^2 dt \right)^q\right] \ll e^{-kq^2 \alpha} T^q (\log x)^{q^2 \alpha}.
$$
\end{proposition}

\begin{proof}
By H\"older's inequality, we have 
\begin{equation}\label{HolderSmallT}
\begin{aligned}
 \ex \left[ \left(\int_T^{2T} |H_{k, \sigma}(t)|^2 dt \right)^q\right] 
 & = \ex  \left[ |H_{k, \sigma}(0)|^{2q(1-q)} \left(\int_T^{2T} |H_{k, \sigma}(t)|^2 |H_{k, \sigma}(0)|^{-2(1-q)}dt \right)^q\right]\\
 & \leq \ex  \left[ |H_{k, \sigma}(0)|^{2q}\right]^{1-q} \ex\left[\int_T^{2T} |H_{k, \sigma}(t)|^2 |H_{k, \sigma}(0)|^{-2(1-q)}dt \right]^q. 
\end{aligned}
\end{equation}
Now it follows from Proposition \ref{MixedRandomMoments} that 
\begin{equation}\label{UpperH0}
\ex  \left[ |H_{k, \sigma}(0)|^{2q}\right] \ll \exp\left( q^2 \sum_{p\leq x^{e^{-k}}} \frac{|g(p)|^2}{p}\right) \ll (e^{-k}\log x)^{q^2 \alpha},
\end{equation}
by \eqref{SelbergOrtho}. Similarly, by Proposition \ref{MixedRandomMoments} together with the fact that $T\leq e^{k}/\log x$ we obtain 
$$ \ex\left[\int_T^{2T} |H_{k, \sigma}(t)|^2 |H_{k, \sigma}(0)|^{-2(1-q)}dt \right]= \int_T^{2T} \ex\left[ |H_{k, \sigma}(t)|^2 |H_{k, \sigma}(0)|^{-2(1-q)} \right] dt \ll T (e^{-k}\log x)^{q^2 \alpha}.$$
Inserting these estimates in \eqref{HolderSmallT} completes the proof.

\end{proof}

\begin{proposition}\label{MediumT}
Let $0<q\leq 1/2$ and $0\leq k\leq K+1$. Assume further that $e^k/\log x \leq T \leq 1$ and that $\sigma\geq -2(k+1)/\log x$. Then we have 
$$ \ex \left[ \left(\int_T^{2T} |H_{k, \sigma}(t)|^2 dt \right)^q\right] \ll  e^{-kq \alpha} (\log x)^{q \alpha} T^{-q^2 \alpha+q+\alpha q}.
$$

\end{proposition}

\begin{proof}
For this range of $T$ we shall split the Euler product $H_{k, \sigma}(t)$ into a product over ``small primes'' ($p\leq \exp(1/T)$) and ``large primes'' ($p> \exp(1/T)$). More specifically we write
$$ H_{k, \sigma}(t)= H^s_{\sigma}(t) H^{\ell}_{k, \sigma}(t), $$
where 
$$
H^s_{\sigma}(t)= \prod_{p\leq e^{1/T}} \left( 1 + \frac{X(p) g(p)}{p^s} + \frac{X(p)^2 g(p^2)}{p^{2 s}} + \dots \right),
$$
and 
$$
H^{\ell}_{k, \sigma}(t)= \prod_{e^{1/T} <p \le x^{e^{-k}}} \left( 1 + \frac{X(p) g(p)}{p^s} + \frac{X(p)^2 g(p^2) }{p^{2 s}} + \dots \right).
$$
Then, similarly to \eqref{HolderSmallT} we have 
\begin{equation}\label{HolderMediumT}
\ex \left[ \left(\int_T^{2T} |H_{k, \sigma}(t)|^2 dt \right)^q\right] 
\leq \ex  \left[ |H^s_{\sigma}(0)|^{2q}\right]^{1-q} \ex\left[\int_T^{2T} |H_{k, \sigma}(t)|^2 |H^s_{ \sigma}(0)|^{-2(1-q)}dt \right]^q,
\end{equation}
and using Proposition \ref{MixedRandomMoments}, we obtain
\begin{equation}\label{UpperSH0}
\ex  \left[ |H^s_{\sigma}(0)|^{2q}\right] \ll \exp\left( q^2 \sum_{p\leq e^{1/T}} \frac{|g(p)|^2}{p}\right) \ll T^{-q^2 \alpha}.
\end{equation}
Now, to bound the second term, we interchange the expectation with the integral, and use the independence of the Euler products over the small and the large primes to get 
$$ \ex\left[\int_T^{2T} |H_{k, \sigma}(t)|^2 |H^s_{ \sigma}(0)|^{-2(1-q)}dt \right]
= \int_{T}^{2T} \ex \left[|H^s_{ \sigma}(t)|^2 |H^s_{ \sigma}(0)|^{-2(1-q)}\right] 
\ex \left[|H^{\ell}_{k, \sigma}(t)|^2\right] dt.$$
By Proposition \ref{MixedRandomMoments}, we have uniformly for $t\in [T, 2T]$ 
$$ \ex \left[|H^s_{ \sigma}(t)|^2 |H^s_{ \sigma}(0)|^{-2(1-q)}\right] \ll T^{-q^2\alpha}. $$ 
Thus, we deduce that 
\begin{equation}\label{UpperLargeH} \ex\left[\int_T^{2T} |H_{k, \sigma}(t)|^2 |H^s_{ \sigma}(0)|^{-2(1-q)}dt \right]
\ll T^{-q^2\alpha} \int_{T}^{2T}
\ex \left[|H^{\ell}_{k, \sigma}(t)|^2\right] dt. 
\end{equation}
We shall now split the range of integration into intervals of length $e^{k}/\log x.$ This gives 
\begin{equation}\label{SplitIntegralLargeT0}
\begin{aligned}
\int_T^{2T}\ex \left[|H^{\ell}_{k, \sigma}(t)|^2\right] dt 
& \leq \sum_{e^{-k} T\log x <n \leq 2e^{-k} T\log x } \int_{|t|\leq e^{k}/(2\log x)} \ex \left[\left |H^{\ell}_{k, \sigma}\left(t+\frac{e^k n}{\log x}\right)\right|^2\right] dt \\
& \ll e^{-k} T\log x \int_{|t|\leq e^{k}/(2\log x)} \ex \left[\left |H^{\ell}_{k, \sigma}\left(t\right)\right|^2\right] dt,
\end{aligned}
\end{equation}
by translation-invariance in law. Moreover, it follows from \eqref{SelbergOrtho} and Proposition \ref{MixedRandomMoments} that uniformly for $t \in [-e^{k}/(2\log x), e^{k}/(2\log x)]$ we have
$$ 
\ex \left[\left |H_{k, \sigma}\left(t\right)\right|^2\right] dt \ll \exp\left( \sum_{e^{1/T}\leq p\leq x^{e^{-k}}} \frac{|g(p)|^2}{p}\right) \ll (Te^{-k}\log x)^{\alpha}.
$$
Combining this bound with \eqref{HolderMediumT}, \eqref{UpperSH0}, \eqref{UpperLargeH}, and \eqref{SplitIntegralLargeT0} completes the proof.
\end{proof}

\begin{proposition}\label{LargeT}
Let $0<q\le 1/2$ and $0\leq k\leq K+1$. Assume further that $T \geq 1$ and that $\sigma\geq -2(k+1)/\log x$. Then we have 
$$ \ex \left[ \left(\int_T^{2T} |H_{k, \sigma}(t)|^2 dt \right)^q\right] \ll e^{-kq \alpha} T^q (\log x)^{q \alpha}.
$$
\end{proposition}

\begin{proof}
We first use H\"older's inequality which implies that
\begin{equation}\label{HolderLargeT}\ex \left[ \left(\int_T^{2T} |H_{k, \sigma}(t)|^2 dt \right)^q\right] \leq \ex \left[ \left(\int_T^{2T} |H_{k, \sigma}(t)|^2 dt \right)\right]^q = \left(\int_T^{2T}\ex \left[|H_{k, \sigma}(t)|^2 \right] dt\right)^q.
\end{equation}
To estimate the inner integral, we split the range of integration into intervals of length $e^{k}/\log x.$ This gives 
\begin{equation}\label{SplitIntegralLargeT}
\begin{aligned}
\int_T^{2T}\ex \left[|H_{k, \sigma}(t)|^2\right] dt 
& \leq \sum_{e^{-k} T\log x <n \leq 2e^{-k} T\log x } \int_{|t|\leq e^{k}/(2\log x)} \ex \left[\left |H_{k, \sigma}\left(t+\frac{e^k n}{\log x}\right)\right|^2\right] dt \\
& \ll e^{-k} T\log x \int_{|t|\leq e^{k}/(2\log x)} \ex \left[\left |H_{k, \sigma}\left(t\right)\right|^2\right] dt,
\end{aligned}
\end{equation}
by translation-invariance in law. Finally, it follows from Proposition \ref{MixedRandomMoments} that uniformly for $t \in [-e^{k}/(2\log x), e^{k}/(2\log x)]$ we have
$$ 
\ex \left[\left |H_{k, \sigma}\left(t\right)\right|^2\right] dt \ll \exp\left( \sum_{p\leq x^{e^{-k}}} \frac{|g(p)|^2}{p}\right) \ll (e^{-k}\log x)^{\alpha}.
$$
Combining this estimate with \eqref{HolderLargeT} and \eqref{SplitIntegralLargeT} completes the proof.

\end{proof}

\begin{proof}[Proof of Theorem \ref{Main}]
Note first that the statement is trivial for $0 \le q \le \frac{1}{\log \log x}$, since in that case, by bounding the $2q$-norm by the $2$-norm and using orthogonality and (\ref{TrivialBoundOnInt}), we have
$$\Psi_{2 q,g}(x) = \Big\Vert \sum_{n \le x} \frac{g(n) X(n)}{\sqrt{n}} \Big\Vert_{2 q}^{2q} \le \Big\Vert \sum_{n \le x} \frac{g(n) X(n)}{\sqrt{n}} \Big\Vert_2^{2q} = \left(\sum_{n\leq x} \frac{|g(n)|^2}{n} \right)^q \ll (\log x)^{B^2 q} \ll 1.
$$

So suppose now that $\frac{1}{\log \log x} \le q \le 1/2$.
The remaining argument, with uniformity in $q$, closely follows the lines of reasoning in \cite{AHZ}.

Let $\sigma= -2(k+1)/\log x$. First, it follows from Propositions \ref{REP} and \ref{REP2} that 
\begin{equation}\label{}
\begin{aligned}
\Psi_{2q,g}(x) 
&\ll \frac{1}{(\log x)^q}\sum_{0\leq k\leq K} \mathbb{E} \Bigg[ \bigg( \int_{\mathbb{R}} \frac{\left| G_{k+1} \left( \frac{1}{2} + \sigma + it \right) \right|^2}{\left| -\sigma + i t \right|^2} \, dt \bigg)^q \Bigg] \\ &+ \sum_{0 \le k \le K+1} e^{- e^k q} \mathbb{E} \left[ |G_k(1/2)|^{2 q} \right] + K \exp \left( -(1+o(1)) q \sqrt{\log x} \right) + 1 \\
& \ll (\log x)^{q^2\alpha} \log (1/q) + \frac{1}{(\log x)^q}\sum_{1\leq k\leq K+1} \mathbb{E} \Bigg[ \bigg( \int_{\mathbb{R}} \frac{\left| H_{k, \sigma} (t) \right|^2}{\left| \sigma + i t \right|^2} \, dt \bigg)^q  \Bigg],
\end{aligned}
\end{equation}
where the last estimate follows from \eqref{UpperH0}, the lower bound on $q$ and the fact that uniformly over $K \ge 1$ and $0 < q \le 1/2$ we have
\[ \sum_{0 \le k \le K+1} e^{-e^k q} \ll \min \{ K, \log(1/q) \}. \]

Let $T_0= -\sigma$, and for $j\geq 1$ we define $T_j= 2 T_{j-1}.$ Then, using that $\sum_{i=0}^{\infty}|y_i|\leq (\sum_{i=1}^{\infty} |y_i|^q)^{1/q}$ for all $\mathbf{y}=(y_1, y_2, \dots)\in \ell^q$ we obtain 
\begin{align*}
 \mathbb{E} \Bigg[ \bigg( \int_{\mathbb{R}} \frac{\left| H_{k, \sigma} (t) \right|^2}{\left| \sigma + i t \right|^2} \, dt \bigg)^q \Bigg]
&\leq 2\mathbb{E} \Bigg[ \bigg( \int_{0}^{T_0} \frac{\left| H_{k, \sigma} (t) \right|^2}{\left| \sigma + i t \right|^2} \, dt \bigg)^q \Bigg] + 2\sum_{j=0}^{\infty} \mathbb{E} \Bigg[ \bigg( \int_{T_j}^{2T_j} \frac{\left| H_{k, \sigma} (t) \right|^2}{\left| \sigma + i t \right|^2} \, dt \bigg)^q \Bigg]\\
& \ll |\sigma|^{-2q} \mathbb{E} \left[ \left( \int_{0}^{T_0} \left| H_{k, \sigma} (t) \right|^2 \, dt \right)^q \right] + \sum_{j=0}^{\infty} \frac{1}{T_j^{2q}} \ex \left[ \left(\int_{T_j}^{2T_{j}} |H_{k, \sigma}(t)|^2 dt \right)^q\right]\\
& \ll \sum_{j=0}^{\infty} \frac{1}{T_j^{2q}} \ex \left[ \left(\int_{T_j}^{2T_{j}} |H_{k, \sigma}(t)|^2 dt \right)^q\right],
\end{align*}
by Proposition \ref{SmallT} combined with the symmetry and translation invariance in law of $H_{k,\sigma}$.
Therefore, we deduce that 
\begin{align}
\Psi_{2q,g}(x) \ll (\log x)^{q^2\alpha} \log (1/q) + \frac{1}{(\log x)^q}\sum_{1 \le k \le K+1}\sum_{j=0}^{\infty} \frac{1}{T_j^{2q}} \ex \left[ \left(\int_{T_j}^{2T_{j}} |H_{k, \sigma}(t)|^2 dt \right)^q\right]. \label{PsiSumMoments}
\end{align}
By Proposition \ref{SmallT}, the contribution of the terms $T_j\leq e^{k}/\log x$ to the inner sum above is 
\begin{equation}\label{ContribSmallT}
\begin{aligned}
\sum_{T_j\leq e^k/\log x}
 \frac{1}{T_j^{2q}} \ex \left[ \left(\int_{T_j}^{2T_{j}} |H_{k, \sigma}(t)|^2 dt \right)^q\right] 
 & \leq  (\log x)^{q^2 \alpha} e^{-kq^2 \alpha} \sum_{ T_j\leq e^k/\log x}
  T_j^{-q} \\
 & \ll k e^{-kq^2 \alpha} (\log x)^{q^2 \alpha+q},
\end{aligned}
\end{equation}
where we used that the number of $T_j \le e^{k}/\log x$ is $\ll k$.
Now, it follows from Proposition \ref{MediumT} that the contribution of the terms $ e^{k}/\log x < T_j \leq 1$ to the inner sum in \eqref{PsiSumMoments} is 
\begin{equation}\label{ContribMediumT}
\begin{aligned}
\sum_{e^k/\log x < T_j \leq 1}
 \frac{1}{T_j^{2q}} \ex \left[ \left(\int_{T_j}^{2T_{j}} |H_{k, \sigma}(t)|^2 dt \right)^q\right] 
 & \leq  (\log x)^{q \alpha}e^{-kq \alpha} \sum_{ e^k/\log x < T_j \leq 1}T_j^{-q^2 \alpha-q(1-\alpha)}\\
 & \ll (\log x)^{q^2 \alpha+q} e^{-(q^2 \alpha+q) k} \sum_{r=0}^{\infty} 2^{-r \left( q^2 \alpha + q (1-\alpha) \right)} \\
 &\ll  \frac{1}{q} (\log x)^{q^2 \alpha+q} e^{-(q^2 \alpha+q) k},
\end{aligned}
\end{equation}
using in the last step that the inner sum is $\le 1/(1-2^{-q(1-\alpha)})\ll 1/q$ since $0 < \alpha < 1$ is fixed.

Finally, Proposition \ref{LargeT} implies that the contribution of the terms $T_j\geq 1$ to the inner sum in \eqref{PsiSumMoments} is 
\begin{equation}\label{ContribLargeT}
\begin{aligned}
\sum_{T_j>1}
 \frac{1}{T_j^{2q}} \ex \left[ \left(\int_{T_j}^{2T_{j}} |H_{k, \sigma}(t)|^2 dt \right)^q\right] 
 & \leq  (\log x)^{q \alpha}  e^{-kq \alpha}  \sum_{T_j > 1}
T_j^{-q} \\
 & \ll \frac{1}{q} (\log x)^{q \alpha} e^{-kq \alpha}.
\end{aligned}
\end{equation}
Inserting the estimates \eqref{ContribSmallT}, \eqref{ContribMediumT} and \eqref{ContribLargeT} in \eqref{PsiSumMoments} we deduce that 
\begin{align*}
\Psi_{2q,g}(x) 
&\ll (\log x)^{q^2\alpha} \log (1/q) + (\log x)^{q^2\alpha}\sum_{1 \le k \le K+1} k e^{-kq^2 \alpha} \\ &+ \frac{1}{q} (\log x)^{q^2\alpha}\sum_{1 \le k \le K+1} e^{-(q^2 \alpha+q) k} + \frac{1}{q} (\log x)^{q\alpha-q} \sum_{1 \le k \le K+1} e^{-kq\alpha}\\
& \ll (\log x)^{q^2\alpha} \left( \log(1/q) + \min \{ K^2,\frac{1}{q^4} \} + \min \{ \frac{K}{q},\frac{1}{q^2} \} + (\log x)^{-q (1 - \alpha) } \min \{ \frac{K}{q},\frac{1}{q^2} \} \right) \\
&\ll \frac{1}{q} (\log x)^{q^2 \alpha} \min \{ (\log \log \log x)^2,\frac{1}{q^3} \},
\end{align*}
which completes the proof.
\end{proof}


\end{document}